\documentclass[12pt, a4paper]{article}
\usepackage{}

\usepackage{amsfonts}
\usepackage{bbm}
\usepackage{mathrsfs}
\usepackage{latexsym}
\usepackage{graphicx}
\usepackage{amssymb}
\usepackage{amsmath}
\usepackage{color,xcolor}
\usepackage{mathtools}

\newtheorem{theorem}{Theorem}[section]
\newtheorem{lemma}[theorem]{Lemma}

\newtheorem{question}[theorem]{Question}

\title{{\Large \bf The Sombor index of trees and unicyclic graphs with given matching number
\thanks{Supported by the National Natural Science Foundation of China (No. 11771443).}~}}
\author{Ting Zhou, Zhen Lin\thanks{Corresponding author. E-mail addresses: tb19080009b1@cumt.edu.cn(T. Zhou) lnlinzhen@163.com (Z. Lin), miaolianying@cumt.edu.cn (L. Miao).}, Lianying Miao \\
{\footnotesize School of Mathematics, China University of Mining and Technology,}\\ {\footnotesize  Xuzhou, 221116, Jiangsu, P.R.
China}\\
}

\date{}
\begin{document}
\openup 1.0\jot
\date{}\maketitle
\begin{abstract}
In 2021, the Sombor index was introduced by Gutman, which is a new degree-based topological molecular descriptors. The Sombor index of a graph $G$ is defined as $SO(G) =\sum_{uv\in E(G)}\sqrt{d^2_G(u)+d^2_G(v)}$, where $d_G(v)$ is the degree of the vertex $v$ in $G$. Let $\mathscr{T}_{n,m}$ and $\mathscr{U}_{n,m}$ be the set of trees and unicyclic graphs on $n$ vertices with fixed matching number $m$, respectively. In this paper, the tree and the unicyclic graph with the maximum Sombor index are determined among $\mathscr{T}_{n,m}$ and $\mathscr{U}_{n,m}$, respectively.

\bigskip

\noindent {\bf MSC Classification:} 05C50, 05C09, 05C90

\noindent {\bf Keywords:} Tree; Unicyclic graph; Sombor index; Matching number
\end{abstract}
\baselineskip 20pt

\section{\large Introduction}

Let $G$ be a simple undirected graph with vertex set $V(G)$ and edge set $E(G)$. For $v\in V(G)$, $N_G(v)$ denotes the set of all neighbors of $v$, and $d_G(v)=|N_G(v)|$ denotes the degree of vertex $v$ in $G$. A pendant vertex of $G$ is a vertex of degree $1$. Let $P_n$ and $C_n$ denote the path and the cycle with $n$ vertices, respectively. Let $T_{n,m}$, shown in Fig. 1.1, be the tree obtained by attaching a pendent vertex to $m-1$ non-central vertices of the star $S_{n-m+1}$, and let $U_{n,m}$, shown in Fig. 1.1, be the unicyclic graph obtained by attaching $n-2m+1$ pendent vertices and $m-2$
paths $P_2$ to one vertex of a cycle $C_3$. It is easy to see that $2\leq m\leq \lfloor\frac{n}{2}\rfloor$ and both of $T_{n,m}$ and $U_{n,m}$ contain a matching with $m$ edges. In particular, they have perfect matching for $n=2m$.

\begin{picture}(300,95)
\put(0,60){\circle*{3}}
\put(0,60){\line(1,0){20}}
\put(20,60){\circle*{3}} \put(17,52){\small $v$}
\put(20,60){\line(5,2){20}}
\put(20,60){\line(5,4){20}}
\put(20,60){\line(3,-4){20}}
\put(40,76){\circle*{3}}
\put(40,76){\line(1,0){20}}
\put(40,68){\circle*{3}}
\put(40,68){\line(1,0){20}}
\put(40,45){\vdots}
\put(60,76){\circle*{3}}
\put(60,68){\circle*{3}}
\put(40,34){\circle*{3}}
\put(40,34){\line(1,0){20}}
\put(60,34){\circle*{3}}
\put(45,48){\tiny  $m-1$}
\put(22,9){$T_{2m,\,m}$}

\put(135,60){\circle*{3}} \put(132,51){\small $v$}
\put(135,60){\line(5,2){20}}
\put(135,60){\line(5,4){20}}
\put(135,60){\line(3,-4){20}}
\put(135,60){\line(-3,4){12}}
\put(135,60){\line(-3,-4){12}}
\put(123,76){\circle*{3}}
\put(123,56){\vdots}
\put(123,44){\circle*{3}}
\put(85,56){\tiny $n-2m+1$}
\put(155,76){\circle*{3}}
\put(155,76){\line(1,0){20}}
\put(155,68){\circle*{3}}
\put(155,68){\line(1,0){20}}
\put(155,45){\vdots}
\put(175,76){\circle*{3}}
\put(175,68){\circle*{3}}
\put(155,34){\circle*{3}}
\put(155,34){\line(1,0){20}}
\put(175,34){\circle*{3}}
\put(160,48){\tiny $m-1$}
\put(135,9){$T_{n,\,m}$ }

  \put(230,60){\circle*{3}}
\put(230,60){\line(1,0){20}}
\put(250,60){\circle*{3}}
\put(250,60){\line(5,2){20}}
\put(250,60){\line(5,4){20}}
\put(250,60){\line(3,-4){20}}
\put(250,60){\line(1,-4){6.9}}
\put(250,60){\line(-1,-4){6.9}}
\put(244,34){\circle*{3}}
\put(244,34){\line(1,0){12}}
\put(256,34){\circle*{3}}
\put(270,76){\circle*{3}}
\put(270,76){\line(1,0){20}}
\put(270,68){\circle*{3}}
\put(270,68){\line(1,0){20}}
\put(270,45){\vdots}
\put(290,76){\circle*{3}}
\put(290,68){\circle*{3}}
\put(270,34){\circle*{3}}
\put(270,34){\line(1,0){20}}
\put(290,34){\circle*{3}}
\put(275,48){\tiny  $m-2$}
\put(252,9){$U_{2m,\,m}$}

\put(365,60){\circle*{3}} \put(132,51){\small $v$}
\put(365,60){\line(5,2){20}}
\put(365,60){\line(5,4){20}}
\put(365,60){\line(3,-4){20}}
\put(365,60){\line(-3,4){12}}
\put(365,60){\line(-3,-4){12}}
\put(365,60){\line(1,-4){6.9}}
\put(365,60){\line(-1,-4){6.9}}
\put(359,34){\circle*{3}}
\put(359,34){\line(1,0){12}}
\put(371,34){\circle*{3}}
\put(353,76){\circle*{3}}
\put(353,56){\vdots}
\put(353,44){\circle*{3}}
\put(313,56){\tiny $n-2m+1$}
\put(385,76){\circle*{3}}
\put(385,76){\line(1,0){20}}
\put(385,68){\circle*{3}}
\put(385,68){\line(1,0){20}}
\put(385,45){\vdots}
\put(405,76){\circle*{3}}
\put(405,68){\circle*{3}}
\put(385,34){\circle*{3}}
\put(385,34){\line(1,0){20}}
\put(405,34){\circle*{3}}
\put(390,48){\tiny $m-2$}
\put(365,9){$U_{n,\,m}$ }

\put(125,-15){Fig 1.1 \quad Graphs $T_{n,\,m}$, $T_{n,\,m}$, $U_{2m,\,m}$, $U_{n,\,m}$. }
\end{picture}

\vskip 5mm

The Sombor index of a graph $G$ is defined as $SO(G) =\sum_{uv\in E(G)}\sqrt{d^2_G(u)+d^2_G(v)}$, which is a new vertex-degree-based molecular structure descriptor was proposed by Gutman \cite{G}. The investigation of the Sombor index of graphs has quickly received much attention. Cruz et al. \cite{CGR} studied the Sombor index of chemical graphs, and characterized the graphs extremal with respect to the Sombor index over the following sets: (connected) chemical graphs, chemical trees, and hexagonal systems. Das et al. \cite{DCC} gave lower and upper bounds on the Sombor index of graphs by using some graph parameters. Moreover, they obtained several relations on Sombor index with the first and second Zagreb indices of graphs. Deng et al. \cite{DTW} obtained a sharp upper bound for the Sombor index among all molecular trees with fixed numbers of vertices, and characterized those molecular trees achieving the extremal value. R\'{e}ti et al. \cite{RDA} characterized graphs with the maximum Sombor index in the classes of all connected unicyclic, bicyclic, tricyclic, tetracyclic, and pentacyclic graphs of a fixed order. Lin et al. \cite{LMZ} obtained lower and upper bounds on the spectral radius, energy and Estrada index of the Sombor matrix of graphs, and characterized the respective extremal graphs. Red\v{z}epovi\'{c} \cite{R} showed that the Sombor index has good predictive potential. For other related results, one may refer to \cite{G1, GA, K, KG, WMLF} and the references therein.

The following question is of interest in graph theory and mathematical chemistry.
\begin{question}\label{que1,1} 
Given a set $\mathcal{G}$ of graphs, find an upper bound for the topological index over
all graphs of $\mathcal{G}$, and characterize the respective extremal graphs.
\end{question}

Inspired by this problem, the topological index of special classes of graphs are well studied in the literature, such as the Zagreb indices of graphs with given clique number \cite{X}, the $ABC$ index of trees with given degree sequence \cite{GLY}, the Randi\'{c} index of trees with given domination number \cite{BNR}, the Estrada index of graphs with given the numbers of cut vertices, connectivity, and edge connectivity \cite{DZX}, etc. Around the Question \ref{que1,1},
we show the following theorem.

\begin{theorem}\label{th1,1} 
If $T\in\mathscr{T}_{n,m}$,  $2\leq m\leq\lfloor\frac{n}{2}\rfloor$, then
$$SO(T)\leq (n-2m+1)\sqrt{(n-m)^2+1}+(m-1)\sqrt{(n-m)^2+4}+(m-1)\sqrt{5}$$
with equality if and only if $T\cong T_{n,m}$.
\end{theorem}

\begin{theorem}\label{th1,2} 
Let $U\in\mathscr{U}_{n,m}$, where $2\leq m\leq\lfloor\frac{n}{2}\rfloor$. Then
$$SO(U)\leq m\sqrt{(n-m+1)^2+4}+(n-2m+1)\sqrt{(n-m+1)^2+1}+\sqrt{5}(m-2)+\sqrt{8}$$
with equality if and only if $U\cong U_{n,m}$.
\end{theorem}

\section{\large  Preliminaries}

The distance $d_G(u,v)$ between two vertices $u$, $v$ of $G$ is the length of one of the shortest $(u,v)$-path in $G$.
A matching $M$ of the graph $G$ is a subset of $E(G)$ such that no two edges in $M$ share a common vertex. A matching $M$ of $G$ is said to be maximum, if for any other matching $M'$ of $G$, $|M'|\leq|M|$. The matching number of $G$ is the number of edges in a maximum matching. If $M$ is a matching of $G$ and vertex $v\in V(G)$ is incident with an edge of $M$, then $v$ is $M$-$saturated$, and if every vertex of $G$ is $M$-saturated, then $M$ is a perfect matching.

\begin{lemma}\label{le2,1} 
Let $a$ and $b$ be integers greater than or equal to one.

{\normalfont (i)} The function $h_1(x)=\sqrt{x^2+a}-\sqrt{(x-1)^2+a}$ is increasing for $x\geq2$.

{\normalfont (ii)} The function $h_2(x)=\sqrt{x^2+b^2}-\sqrt{x^2+(b-1)^2}$ is decreasing for $x\geq2$.
\end{lemma}

\begin{proof}
Calculating the derivative of $h_1(x)$ and $h_2(x)$, we have the proof. $\Box$
\end{proof}

\begin{lemma}{\bf (\cite{KS})}\label{le2,2} 
Let $U\subseteq \mathbb{R}$ be an open interval and $f:U\rightarrow U$ a convex function. Let $a_1\geq a_2\geq\ldots\geq a_n$ and $b_1\geq b_2\geq\ldots\geq b_n$ be such elements in $U$ that inequalities $a_1+a_2+\ldots+a_n\geq b_1+b_2+\ldots+b_i$ hold for every $i\in\{1,2,\ldots,n\}$ and equality holds for $i=n$. Then, $f(a_1)+f(a_2)+\cdots+f(a_n)\geq f(b_1)+f(b_2)+\cdots+f(b_n)$.
\end{lemma}

\begin{lemma}\label{le2,3} 
Let $u_0v_0\in E(G)$, $N_G(u_0)\cap N_G(v_0)=\Phi$, and denote by $u_1,u_2,\ldots,u_s$ and $v_1,v_2,\ldots,v_t$ the neighbors of $u_0$ and $v_0$, respectively, where $s,t\geq1$. Let $G'$ be a new graph with vertex set $V(G')=V(G)$ and edge set $E(G')=E(G)-\{v_0v_i|1\leq i\leq t\}+\{u_0v_i|1\leq i\leq t\}$. Then $SO(G')>SO(G)$.
\end{lemma}

\begin{proof}
According to the given conditions, we have $d_{G}(u_0)=s+1$, $d_{G'}(u_0)=s+t+1$, $d_{G}(v_0)=t+1$, $d_{G'}(v_0)=1$ and $d_{G}(u_i)=d_{G'}(u_i)$ for $1\leq i\leq s$, $d_{G}(v_j)=d_{G'}(v_j)$ for $1\leq j\leq t$. Since $f(x)=x^2$ is a convex function, by Lemma \ref{le2,2}, $f(s+t+1)+f(1)>f(s+1)+f(t+1)$. Thus we have
\begin{align*}
 & SO(G')-SO(G) \\
= {}& \sum_{uv\in E(G')}\sqrt{d^2_{G'}(u)+d^2_{G'}(v)}-\sum_{uv\in E(G)}\sqrt{d^2_{G}(u)+d^2_{G}(v)}\\
= {} & \sum_{i=1}^s\sqrt{d^2_{G'}(u_0)+d^2_{G'}(u_i)}+\sum_{j=1}^t\sqrt{d^2_{G'}(u_0)+d^2_{G'}(v_j)}+\sqrt{d^2_{G'}(u_0)+d^2_{G'}(v_0)}\\
{} & -\Big[\sum_{i=1}^s\sqrt{d^2_{G}(u_0)+d^2_{G}(u_i)}+\sum_{j=1}^t\sqrt{d^2_{G}(v_0)+d^2_{G}(v_j)}+\sqrt{d^2_{G}(u_0)+d^2_{G}(v_0)}\Big]\\
= {}& \sum_{i=1}^s\!\sqrt{(s\!+\!t\!+\!1)^2+d^2_{G'}(u_i)}+\sum_{j=1}^t\!\sqrt{(s\!+\!t\!+\!1)^2+d^2_{G'}(v_j)}+\!\sqrt{(s\!+\!t\!+\!1)^2+1}\\
{}& -\Big[\sum_{i=1}^s\!\sqrt{(s\!+\!1)^2+d^2_{G}(u_i)}+\sum_{j=1}^t\!\sqrt{(t\!+\!1)^2+d^2_{G}(v_j)}+\!\sqrt{(s\!+\!1)^2+(t\!+\!1)^2}\Big]\\
 > {}& \sqrt{(s+t+1)^2+1}-\sqrt{(s+1)^2+(t+1)^2}\\
 >{} & 0.
\end{align*}
This completes the proof. $\Box$
\end{proof}

\section{\large The proof of Theorem \ref{th1,1}}

\begin{lemma}{\bf (\cite{HL})}\label{le3,1} 
Let $T\in\mathscr{T}_{2m,m}$, where $m\geq2$. Then there exists a pendant vertex of $T$ whose unique neighbor is of degree two.
\end{lemma}

\begin{lemma}{\bf (\cite{BG, HL})}\label{le3,2} 
Let $T\in\mathscr{T}_{n,m}$, where $n>2m$. Then there exists a maximum matching $M$ and a pendant vertex $u$ of $T$ such that $u$ is not $M$-saturated.
\end{lemma}

\begin{lemma}\label{le3,3} 
Let $m\geq2$ and $T\in \mathscr{T}_{2m,m}$. Then
$$\sqrt{8}(n-3)+2\sqrt{5}\leq SO(T)\leq \sqrt{m^2+1}+(m-1)\sqrt{m^2+4}+\sqrt{5}(m-1).$$
The equality in the left hand side holds if and only if $T\cong P_{2m}$, and the equality in the right hand side holds if and only if $T\cong T_{2m,m}$.
\end{lemma}

\begin{proof}Since $P_{2m}\in\mathscr{T}_{2m,m}$, from \cite{G}, the path $P_n$ attains the minimum Sombor index among all connected graphs of order $n$, so $P_{2m}$ is the minimum Sombor index among $\mathscr{T}_{2m,m}$.

Now, we consider the right side inequality. Let $f(m)=\sqrt{m^2+1}+(m-1)\sqrt{m^2+4}+\sqrt{5}(m-1)$. By direct calculation, we have
$SO(T_{2m,m})=f(m)$. Therefore, we next prove $SO(T)\leq SO(T_{2m,m})$ by induction on $m$. If $m=2$, it is east to check that $T=T_{4,2}=P_4$, the result holds.

Suppose that $m\geq3$ and the result holds for trees in $\mathscr{T}_{2m-2,m-1}$. Let $T\in\mathscr{T}_{2m,m}$ with a perfect matching $M$. By Lemma \ref{le2,1}, there exists a pendant vertex $u_0$ in $T$ adjacent to a vertex $u$ of degree two. Then $uu_0\in M$ and $T-\{u, u_0\}\in \mathscr{T}_{2m-2,m-1}$. Let $v$ be the neighbor of $u$ different from $u_0$. Denote by $u, v_1, \ldots,v_{s-1}$ the neighbors of $v$ in $T$, that is $s=d_T(v)$. Let $r$ be the number of pendant neighbors of $v$. Then $r=0$ or $r=1$. Since $T$ has a perfect matching, every pendant vertex is $M$-saturated, we have $d_T(v)\leq m$.
By lemma \ref{le2,1} and the induction hypothesis, we have
\begin{eqnarray*}
SO(T) & = & SO(T-\{u,u_0\})+\sqrt{5}+\sqrt{d^2_T(v)+4}+r\left[\sqrt{d^2_T(v)+1}-\sqrt{(d_T(v)-1)^2+1}\right]\\
&  & +\sum_{i=r}^{s-1}\left[\sqrt{d^2_T(v_i)+d^2_T(v)}-\sqrt{d^2_T(v_i)+(d_T(v)-1)^2}\right]\\
& = & SO(T-\{u,u_0\})+\sqrt{5}+\sqrt{s^2+4}+r\left[\sqrt{s^2+1}-\sqrt{(s-1)^2+1}\right]\\
&  & +\sum_{i=r}^{s-1}\left[\sqrt{d^2_T(v_i)+s^2}-\sqrt{d^2_T(v_i)+(s-1)^2}\right]\\
& \leq & f(m-1)+\sqrt{5}+\sqrt{m^2+4}+\sqrt{m^2+1}-\sqrt{(m-1)^2+1}\\
&  & + (m-2)\left[\sqrt{m^2+4}-\sqrt{(m-1)^2+4}\right]\\
& = & f(m)
\end{eqnarray*}
with equality if and only if $T-\{u, u_0\}\cong T_{2m-2,\,m-1}$, $d_T(v)=s=m$, $r=1$ and $d_T(v_i)=2$ for $2\leq i \leq s-1$, that is to say
$T\cong T_{2m,m}$. This completes the proof. $\Box$
\end{proof}

\noindent{\bf Proof of Theorem \ref{th1,1}} Let $f(n,m)=(n-2m+1)\sqrt{(n-m)^2+1}+(m-1)\sqrt{(n-m)^2+4}$\\$+(m-1)$.
By direct calculation, we have $SO(T_{n,m})=f(n,m)$.
Therefore, we next prove $SO(T)\leq SO(T_{n,m})$ by induction on $n$. If $n=2m$, the result follows from Lemma \ref{le3,3}.

Suppose that $n>2m$ and the result holds for trees in $\mathscr{T}_{n-1,m}$. Let $T\in\mathscr{T}_{n,m}$. By Lemma \ref{le3,2}, there is a maximum matching $M$ and a pendant vertex $u$ of $T$ such that $u$ is not $M$-saturated. Then $T-u\in\mathscr{T}_{n-1,m}$. Let $v$ be the unique neighbor of $u$. Since $M$ is a maximum matching, $M$ contains one edge incident with $v$. Note that there are $n-1-m$ edges of $T$ outside $M$. So $d_T(v)-1\leq n-1-m$, that is $d_T(v)\leq n-m$. Denote by $u, v_1, \ldots,v_{s-1}$ the neighbors of $v$ in $T$, where $s=d_T(v)$. Let $r$ be the number of pendant neighbors of $v$ in $T$, where $1\leq r\leq d_T(v)-1$. Note that at least $r-1$ pendant vertices of $v$ are not $M$-saturated, and there are $n-2m$ vertices are not $M$-saturated in $T$, then $r\leq n-2m+1$. By Lemma \ref{le2,1} and induction hypothesis, we have
\begin{eqnarray*}
\nonumber SO(T) & = & SO(T-u)+\sqrt{d^2_T(v)+1}+(r-1)\left[\sqrt{d^2_T(v)+1}-\sqrt{(d_T(v)-1)^2+1}\right]\\
& & + \sum_{i=r}^{s-1}\left[\sqrt{d^2_T(v_i)+d^2_T(v)}-\sqrt{d^2_T(v_i)+(d_T(v)-1)^2}\right]\\
& = & SO(T-{u})+\sqrt{s^2+1}+(r-1)\left[\sqrt{s^2+1}-\sqrt{(s-1)^2+1}\right]\\
&  & + \sum_{i=r}^{s-1}\left[\sqrt{d^2_T(v_i)+s^2}-\sqrt{d^2_T(v_i)+(s-1)^2}\right]\\
& \leq & f(n-1,m)+\sqrt{(n-m)^2+1}+(n-2m)\left[\sqrt{(n-m)^2+1}-\sqrt{(n-m-1)^2+1}\right]\\
& & + (m-1)\left[\sqrt{(n-m)^2+4}-\sqrt{(n-m-1)^2+4}\right]\\
& = & (n-2m+1)\sqrt{(n-m)^2+1}+(m-1)\sqrt{(n-m)^2+4}+(m-1)\\
& = & f(n,m)
\end{eqnarray*}
with equalities if and only if $G-u\cong T_{n-1,m}$, $d_T(v)=s=n-m$, $r=n-2m+1$ and $d_T(v_i)=2$ for $r\leq i \leq s-1$, that is to say
$T\cong T_{n,m}$. This completes the proof.  $\Box$

\section{\large The proof of Theorem \ref{th1,2}}

For a unicyclic graph $U$, we assume that the cycle $C_k=u_1u_2\ldots u_ku_1$ is the unique cycle of $U$. Let $T_i=T_{u_i}$ be the tree component containing $u_i$ in $U-E(C_k)$.

\begin{lemma}{\bf (\cite{CT})}\label{le4,1} 
Let $U\in\mathscr{U}_{2m,m}$, where $m\geq3$. If $u\in V(T_i)$ is a pendant vertex that is furthest from the root $u_i$ such that $d_U(u,u_i)\geq2$, then unique neighbor of $u$ is of degree two.
\end{lemma}

\begin{lemma}{\bf (\cite{YT})}\label{le4,2} 
Let $U\in\mathscr{U}_{n,m}$, where $n>2m$, and $U\neq C_n$. Then there exists a maximum matching $M$ and a pendant vertex $u$ of $T$ such that $u$ is not $M$-saturated.
\end{lemma}

\begin{lemma}\label{le4,3} 
Let $U\in \mathscr{U}_{2m,m}$, where $m\geq2$. Then
$$2\sqrt{8}m\leq SO(U)\leq m\sqrt{(m+1)^2+4}+\sqrt{(m+1)^2+1}+\sqrt{5}(m-2)+\sqrt{8}.$$
The equality in the left holds if and only if $U\cong C_{2m}$, and the equality in the right holds if and only if $U\cong U_{2m,m}$.
\end{lemma}

\begin{proof} Since $C_{2m}\in\mathscr{U}_{2m,m}$, from \cite{RDA}, the cycle $C_n$ attains the minimum Sombor index among all unicyclic graphs of order $n$, so $C_{2m}$ is the minimum Sombor index among $\mathscr{U}_{2m,m}$.

Now, we consider the right side inequality. Let $g(m)=m\sqrt{(m+1)^2+4}+\sqrt{(m+1)^2+1}$\\$+\sqrt{5}(m-2)+\sqrt{8}$. By direct calculation, we have
$SO(U_{2m,m})=g(m)$. Therefore, we next prove $SO(U)< SO(U_{2m,m})$ for any $U\in\mathscr{U}_{2m,m}\backslash \{U_{2m,m}\}$. If $m=2$, it is easy to check that $U\in\{U_{4,2},C_4\}$, $SO(C_4)<SO(U_{4,2})$, the result holds.

Next, suppose that $m\geq3$ and $U$ has the maximum Sombor index among $\mathscr{U}_{2m,m}$. Let $C_k=u_1u_2\ldots u_ku_1$ be the unique cycle of $U$, we consider the following two cases:

{\bf Case 1.}  For any pendant vertex $u$ of $U$, if $u\in V(T_i)$, then $d_U(u,u_i)=1$. In this case, $U$ is a graph obtained from $C_k$ by attaching some pendant vertices to $u_i$ for $1\leq i\leq k$, where $m\leq k<2m$. Since every pendant vertex in $U$ is $M$-saturated, we have $d_U(u_i)\leq3$ for $1\leq i\leq k$. If $m<k$, then there exists at least one edge $u_iu_{i+1(mod\ k)}$ of $C_k$ such that $u_iu_{i+1(mod\ k)}\in M$. Let $U'=U-u_iu_{i+1(mod\ k)}+u_iu_{i+2(mod\ k)}$, by Lemma \ref{le2,3}, $SO(U)<SO(U')$. This contradicts to the choice of $U$. Therefore, $m=k$, that is to say, $U$ is the graph attaching a pendant vertex to each vertex of $C_m$. So $SO(U)=\sum_{uv\in E(U)}\sqrt{d^2_U(u)+d^2_U(v)}=(\sqrt{18}+\sqrt{10})m<g(m)$.

{\bf Case 2.} There exists a pendant vertex $u$ of $U$, $u\in V(T_i)$ such that $d_U(u,u_i)\geq2$. In this case, let $v$ be the unique neighbor of $u$, by Lemma \ref{le4,1}, $d_U(v)=2$. We prove our results by induction on $m$. Suppose that $m\geq3$ and the result holds for unicyclic graphs in $\mathscr{U}_{2m-2,m-1}$. Since $uv\in M$ and $U-\{u,v\}\in \mathscr{U}_{2m-2,m-1}$. Let $w$ be the neighbor of $v$ different from $u$. Denote by $v, w_1, \ldots,w_{s-1}$ the neighbors of $w$ in $U$, where $s=d_T(w)$. Let $r$ be the number of pendant neighbors of $w$. Then $r=0$ or $r=1$. Since $U$ has a perfect matching, every pendant vertex is $M$-saturated, we have $d_T(w)=s\leq m+1$. By lemma \ref{le2,1} and the induction hypothesis, we have
\begin{eqnarray*}
\nonumber SO(U) & = & SO(U-\{u,v\})+\sqrt{5}+\sqrt{d^2_U(w)+4}+r\Big[\sqrt{d^2_U(w)+1}-\sqrt{(d_U(w)-1)^2+1}\Big]\\
& & +\sum_{i=r+1}^{s-1}\Big[\sqrt{d^2_U(w_i)+d^2_U(w)}-\sqrt{d^2_U(w_i)+(d_U(w)-1)^2}\Big]\\
& = & SO(U-\{u,v\})+\sqrt{5}+\sqrt{s^2+4}+r\Big[\sqrt{s^2+1}-\sqrt{(s-1)^2+1}\Big]\\
& & +\sum_{i=r+1}^{s-1}\Big[\sqrt{d^2_U(w_i)+s^2}-\sqrt{d^2_U(w_i)+(s-1)^2}\Big]\\
& \leq & g(m-1)+\sqrt{5}+\sqrt{(m+1)^2+4}+\sqrt{(m+1)^2+1}-\sqrt{m^2+1}\\
& & +(m-1)\Big[\sqrt{(m+1)^2+4}-\sqrt{m^2+4}\Big]\\
& = & g(m)
\end{eqnarray*}
with equalities if and only if $U-\{u,v\}\cong U_{2m-2,m-1}$, $r=1$, $d_U(w_1)=1$, $d_U(w_i)=2$ for $2\leq i\leq s-1$ and $d_U(w)=s=m+1$, that is to say
$U\cong U_{2m,m}$. This completes the proof. $\Box$
\end{proof}

\noindent{\bf Proof of Theorem \ref{th1,2}} Let $g(n,m)=m\sqrt{(n-m+1)^2+4}+(n-2m+1)\sqrt{(n-m+1)^2+1}$\\$+\sqrt{5}(m-2)+\sqrt{8}$.
By direct calculation, we have $SO(U_{n,m})=g(n,m)$. Therefore, we next prove $SO(U)\leq SO(U_{n,m})$ by induction on $n$. If $n=2m$, the result follows from Lemma \ref{le3,3}.

Suppose that $n>2m$ and the result holds for unicyclic graphs in $\mathscr{U}_{n-1,m}$. Let $U\in\mathscr{U}_{n,m}$. By Lemma \ref{le4,2}, there is a maximum matching $M$ and a pendant vertex $u$ of $U$ such that $u$ is not $M$-saturated. Then $U-u\in\mathscr{U}_{n-1,m}$. Let $v$ be the unique neighbor of $u$, $d_U(v)\leq n-m+1$. Denote by $u, v_1, \ldots,v_{s-1}$ the neighbors of $v$ in $U$, where $s=d_U(v)$. Let $r$ be the number of pendant neighbors of $v$ in $U$, where $1\leq r\leq d_U(v)-1$. Note that at least $r-1$ pendant vertices of $v$ are not $M$-saturated, and there are $n-2m$ vertices are not $M$-saturated in $U$, then $r\leq n-2m+1$. By Lemma \ref{le2,1} and induction hypothesis, we have
\begin{eqnarray*}
SO(U) & = & SO(U-u)+(r-1)\Big[\sqrt{d^2_U(v)+1}-\sqrt{(d_U(v)-1)^2+1}\Big]\\
& & +\sum_{i=r}^{s-1}\Big[\sqrt{d^2_U(v_i)+d^2_U(v)}-\sqrt{d^2_U(v_i)+(d_U(v)-1)^2}\Big]+\sqrt{d^2_U(v)+1}\\
& = & SO(U-{u})+(r-1)\Big[\sqrt{s^2+1}-\sqrt{(s-1)^2+1}\Big]\\
& & +\sum_{i=r}^{s-1}\Big[\sqrt{d^2_U(v_i)+s^2}-\sqrt{d^2_U(v_i)+(s-1)^2}\Big]+\sqrt{s^2+1}\\
& \leq & g(n-1,m)+(n-2m)\Big[\sqrt{(n-m+1)^2+1}-\sqrt{(n-m)^2+1}\Big]\\
& & +m\Big[\sqrt{(n-m+1)^2+4}-\sqrt{(n-m)^2+4}\Big]+\sqrt{(n-m+1)^2+1}\\
& = & g(n,m)
\end{eqnarray*}
with equalities if and only if $U-u\cong U_{n-1,m}$, $r=n-2m+1$, $d_U(v_i)=2$ for $r\leq i\leq s-1$ and $d_U(v)=s=n-m+1$, that is to say
$U\cong U_{n,m}$. This completes the proof. $\Box$

\vskip 5mm

\small {

}

\end{document}